\documentclass[12pt]{article}
\usepackage{latexsym, amssymb, amsthm, amsmath}
\usepackage{dsfont}

\DeclareMathAlphabet\gothic{U}{euf}{m}{n}

\makeatletter
\def\eqnarray{\stepcounter{equation}\let\@currentlabel=\theequation
\global\@eqnswtrue
\tabskip\@centering\let\\=\@eqncr
$$\halign to \displaywidth\bgroup\hfil\global\@eqcnt\z@
  $\displaystyle\tabskip\z@{##}$&\global\@eqcnt\@ne
  \hfil$\displaystyle{{}##{}}$\hfil
  &\global\@eqcnt\tw@ $\displaystyle{##}$\hfil
  \tabskip\@centering&\llap{##}\tabskip\z@\cr}

\def\endeqnarray{\@@eqncr\egroup
      \global\advance\c@equation\m@ne$$\global\@ignoretrue}

\def\@yeqncr{\@ifnextchar [{\@xeqncr}{\@xeqncr[5pt]}}
\makeatother

\allowdisplaybreaks

\newtheorem{lemma}{Lemma}[section]
\newtheorem{thm}[lemma]{Theorem}
\newtheorem{cor}[lemma]{Corollary}
\newtheorem{prop}[lemma]{Proposition}

\theoremstyle{definition}

\newcommand{\gotb}{\gothic{b}}

\newcounter{teller}

\newenvironment{tabel}{\begin{list}%
{\rm  (\alph{teller})\hfill}{\usecounter{teller} \leftmargin=1.1cm
\labelwidth=1.1cm \labelsep=0cm \parsep=0cm}
                      }{\end{list}}

\newcounter{tellerr}

\newcounter{tellerrr}

\newcounter{proofstep}

\newcommand{\Ni}{\mathds{N}}

\newcommand{\Ri}{\mathds{R}}
\newcommand{\R}{\mathds{R}}
\newcommand{\Ci}{\mathds{C}}

\newcommand{\divv}{\mathop{\rm div}}

\newcommand{\RRe}{\mathop{\rm Re}}

\newcommand{\sgn}{\mathop{\rm sgn}}

\newcommand{\one}{\mathds{1}}

\newcommand{\ca}{{\cal A}}
\newcommand{\cb}{{\cal B}}

\newcommand{\altnorm}[1]{{\left\vert\kern-0.25ex\left\vert\kern-0.25ex\left\vert #1 
    \right\vert\kern-0.25ex\right\vert\kern-0.25ex\right\vert}}

\begin{document}
\thispagestyle{empty}

\vspace*{1cm}
\begin{center}
{\large\bf $L^\infty$-estimates for the Neumann problem on general domains} \\[5mm]
\large  A.F.M. ter Elst, H. Meinlschmidt and J. Rehberg

\end{center}

\vspace{5mm}

\begin{list}{}{\leftmargin=1.8cm \rightmargin=1.8cm \listparindent=10mm 
   \parsep=0pt}
\item
\small
{\sc Abstract}.
Let $\Omega \subset \Ri^d$ be bounded open and connected. 
Suppose that $W^{1,2}(\Omega) \subset L^r(\Omega)$ for some $r > 2$.
Let $A$ be a pure second-order elliptic differential operator
with bounded real measurable coefficients on $\Omega$.
Let $q > d$ with $\frac{1}{2}-\frac{1}{q} > \frac{1}{r}$.
If $p$ is the dual exponent of~$q$, then we show that the pre-image 
of the space $(W^{1,p}(\Omega))^*$ under the map $A$ is contained in the space of 
bounded functions on $\Omega$. 
The considerations are
complemented by results on optimal Sobolev regularity for~$A$.
\end{list}

\let\thefootnote\relax\footnotetext{
\begin{tabular}{@{}l}
{\em AMS Subject Classification}. 35D10, 46E35.\\
{\em Keywords}. Regularity, Neumann problem, elliptic operators.
\end{tabular}}

\section{Introduction} \label{boundedS1}

If $A$ is a pure second-order elliptic operator with Dirichlet boundary 
conditions and real measurable coefficients on a bounded connected open 
set $\Omega \subset \Ri^d$, then it is not too hard to show that the 
resolvent operator $A^{-1}$ maps $L^q(\Omega)$ into $L^\infty(\Omega)$
if $q > d$. 
It is a famous result of Stampacchia (\cite{Stam2} Theorem~4.4)
that $A^{-1}$ extends to a continuous operator from $W^{-1,q}(\Omega)$ into $L^\infty(\Omega)$
if $q > d$.
Inspecting the proof, it is becomes
clear that this result extends without difficulties from the pure Dirichlet case 
to the case of mixed boundary conditions,
as long as the Dirichlet part of the boundary is large enough to imply a
Poincar\'e inequality.
It is also possible to extend the result to merely Neumann boundary
conditions if a positive scalar is added to the operator, and hence the resulting 
operator is coercive.
The idea how to do this can be found in the book of Tr\"oltzsch (\cite{Troltzsch}
Section~7.2.2).
What remains open is the pure divergence form operator with
pure Neumann boundary condition. 
It is clear that a `naive' generalisation 
cannot work, since one can add to any solution of such a Neumann problem
an arbitrary constant and again obtains a solution.

The main theorem of this paper is as follows.

\begin{thm} \label{tbounded101}
Let $\Omega \subset \Ri^d$ be a bounded connected open set.
Let $r \in (2,\infty)$ and suppose that $W^{1,2}(\Omega) \subset L^r(\Omega)$.
Let $\mu \colon \Omega \to \Ri^{d \times d}$ be a bounded measurable 
function.
Suppose there exists a $\nu > 0$ such that 
\[
\RRe \sum_{k,\ell=1}^d \mu(x) \, \xi_k \, \overline{\xi_\ell} 
\geq \nu \, |\xi|^2
\]
for all $\xi \in \Ci^d$ and almost all $x \in \Omega$.
Define $\ca \colon W^{1,2}(\Omega) \to (W^{1,2}(\Omega))^*$ by 
\[
\ca(u,v) 
= \int_\Omega \mu \nabla u \cdot \overline{\nabla v}
 .  \]
Let $q \in (d,\infty)$ and suppose that $\frac{1}{2}-\frac{1}{q} > \frac{1}{r}$. 
If $u \in W^{1,2}(\Omega)$
with $\ca u \in (W^{1,p}(\Omega))^*$, where $p$ is the dual exponent of $q$, then $u \in L^\infty(\Omega)$. 

More precisely, for all
$T \in (W^{1,p}(\Omega))^*$ with $T(\one) = 0$ there is a unique
$u \in W^{1,2}(\Omega)$ with $\int_\Omega u = 0$ satisfying $\ca u = T$. 
Moreover, there exists a $c > 0$ independent of $T$ such that
$\|u\|_{L^\infty(\Omega)} \leq c \, \|T\|_{(W^{1,p}(\Omega))^*}$.
\end{thm}

We emphasise that the Sobolev embedding $W^{1,2}(\Omega) \subset L^r(\Omega)$ assumption is
very weak.
If $2^*$ is the first Sobolev exponent, that is
$\frac{1}{2^*} = \frac{1}{2} - \frac{1}{d}$, then it follows from 
scaling that $r \leq 2^*$.
The assumption $\frac{1}{2}-\frac{1}{q} > \frac{1}{r}$ implies that $q > d$.
It is well known that there is a connection between Sobolev embeddings and
solvability of Neumann problems.
We exemplarily refer to Maz'ya and
Poborchi{\u{\i}}~\cite{MazyaPoborchii1,MazyaPoborchii2}
and~\cite{MazED2}, Section~6.10.

If $d \geq 3$, then the optimal case in our assumption is $r = 2^*$, the first Sobolev exponent.
Then the condition $\frac{1}{2}-\frac{1}{q} > \frac{1}{r}$ is merely the 
condition $q > d$, as in the Stampacchia theorem for the Dirichlet boundary condition.
This optimal assumption is satisfied for example
by any open bounded set which is the finite union of connected $W^{1,2}$-extension domains,
such as for example Lipschitz domains.
Another example is that of a connected John domain (\cite{Bojarski}, Section~6).
If the domain has cusps, then the
full Sobolev embedding is usually not available, but the embedding 
$W^{1,2}(\Omega) \subset L^r(\Omega)$ still holds for some $r \in (2,2^*)$ if the cusps are of polynomial
type by~\cite{AF}, Theorem~4.51.
We also refer to Maz'ya~\cite{MazED2}, Section~6.9 for more
geometric conditions.
It is also known that the embedding cannot
hold true for any $r > 2$ if the boundary of $\Omega$ has cusps of exponential
sharpness, see~\cite{AF} Theorem~4.48.
Note that in the case of Dirichlet boundary
conditions, one always has the optimal embeddings
$W^{1,2}_0(\Omega) \subset L^{2^*}(\Omega)$ if $d \geq 3$,
and $W^{1,2}_0(\Omega) \subset L^r(\Omega)$ for all $r \in (2,\infty)$ if $d=2$.

The proof of Theorem~\ref{tbounded101} follows the ideas of Stampacchia and
uses truncations of Sobolev functions.
It relies on the Stampacchia lemma (\cite{KinS}
Chapter~II, Appendix~B, Lemma~2.1) and at its heart lies a uniform estimation of the 
Poincar\'{e} constants of the truncations of mean value free Sobolev functions,
Lemma~\ref{lbounded207} below.

We also prove that the pure Neumann operator $\ca$ admits optimal Sobolev regularity in the
setting of Theorem~\ref{tbounded101} for $q$ sufficiently close to~$2$.
This means that the
domain of the part of the operator
$\ca$ in $(W^{1,p}(\Omega))^*$ coincides with $W^{1,q}_\perp(\Omega)$, the mean
value free functions in $W^{1,q}(\Omega)$, where again $p$ is the dual exponent to $q$.
The result relies
on interpolation and the {\v{S}}ne{\u{\i}}berg stability theorem.
We refer to Theorem~\ref{tbounded402} below.

The outline of this paper is as follows.
 In Section~\ref{boundedS2} we show that a Sobolev
embedding implies a Poincar\'e inequality on any $L^p$-space.
 We use this in
Section~\ref{boundedS3} to adapt the argument of Stampacchia to deduce the boundedness as
stated in Theorem~\ref{tbounded101}.
In Section~\ref{boundedS4} we derive optimal Sobolev
regularity results for $\ca$ and some consequences of these based on the results in Section~\ref{boundedS2}.

\smallskip

We conclude with an example.
We formally attach the following boundary value problem to the equation $\ca u = T$ with 
$T \in (W^{1,p}(\Omega))^*$
as in Theorem~\ref{tbounded101}:
\begin{eqnarray*}
  - \divv(\mu \nabla u) & = & f \quad \text{in}~\Omega, \\*
  - n \cdot \mu\nabla u & = & g \quad \text{on}~\partial\Omega,
\end{eqnarray*}
where $f \in L^s(\Omega)$ and
$g \in L^t(\partial\Omega;\mathcal{H}_{d-1})$ for appropriate
values of $s$ and $t$, where $n$ is the normal. 
Since $T$ is only supposed to be a functional on $W^{1,p}(\Omega)$,
inhomogeneous boundary data is allowed. For the foregoing
boundary value problem, $T$ takes the form
\begin{equation*}
  T(v) = \int_\Omega f\,\overline{v} 
  + \int_{\partial\Omega} g\,\overline{\tau v} \, \mathrm{d}\mathcal{H}_{d-1},
\end{equation*}
where $\tau$ is the trace operator onto $\partial\Omega$. If the
domain $\Omega$ is sufficiently regular to allow the
application of the divergence theorem and to admit a suitable
trace operator, this
formulation and its connection to $\ca u = T$ can be made
rigorous, see Ciarlet (\cite{Cia},
Chapter~1.2) or~\cite{GGZ}, Chapter~2.2. A particular case would
be that of a Lipschitz graph domain $\Omega$.

\section{Sobolev and Poincar\'e} \label{boundedS2}

We first show that a Sobolev type embedding extrapolates 
to compactness of the inclusion map $W^{1,p}(\Omega) \subset L^p(\Omega)$.

\begin{lemma} \label{lbounded201}
Let $\Omega \subset \Ri^d$ be open and bounded.
Let $q \in (1,\infty)$ and suppose there exists a $\delta > 0$ such that 
$W^{1,q}(\Omega) \subset L^{q+\delta}(\Omega)$.
Let $p \in (1,\infty)$.
Then the inclusion $W^{1,p}(\Omega) \subset L^p(\Omega)$ is compact.
Moreover, there exists a $\delta' > 0$ such that 
$W^{1,p}(\Omega) \subset L^{p+\delta'}(\Omega)$.
\end{lemma}
\begin{proof}
We show that there exists an $s > p$ such that 
$W^{1,p}(\Omega) \subset L^s(\Omega)$. 
Then the compactness of the inclusion $W^{1,p}(\Omega) \subset L^p(\Omega)$ 
follows as in~\cite{Daners7} Lemma~7.1.
Suppose that $p \in (1,q)$ (the case $p \in (q,\infty)$ is similar).
Fix $r \in (1,p)$.
It follows from Liu--Tai~\cite{LiuTai} Theorem~9 that the real interpolation space
$(W^{1,1}(\Omega),W^{1,\infty}(\Omega))_{1 - \frac{1}{t},t} = W^{1,t}(\Omega)$
for all $t \in (1,\infty)$. 
Here $W^{1,\infty}(\Omega)$ is the Sobolev space of all
$L^\infty(\Omega)$ functions whose weak partial derivatives are also $L^\infty(\Omega)$ functions.
Let $\theta \in (0,1)$ be such that $\frac{1}{p} = \frac{1-\theta}{r} + \frac{\theta}{q}$.
Then by complex interpolation
\begin{eqnarray}
\bigl[W^{1,r}(\Omega),W^{1,q}(\Omega)\bigr]_\theta
& = & \Bigl[\bigl(W^{1,1}(\Omega),W^{1,\infty}(\Omega)\bigr)_{1 - \frac{1}{r},r},
       \bigl(W^{1,1}(\Omega),W^{1,\infty}(\Omega)\bigr)_{1 - \frac{1}{q},q}\Bigr]_\theta \label{ebounded1} \\
& = & \bigl(W^{1,1}(\Omega),W^{1,\infty}(\Omega)\bigr)_{1 - \frac{1}{p},p}
= W^{1,p}(\Omega) \notag
,
\end{eqnarray}
where we used the reiteration theorem~\cite{BL} Theorem~4.7.2 in the second step.
The inclusions $W^{1,r}(\Omega) \to L^r(\Omega)$ and 
$W^{1,q}(\Omega) \to L^{q + \delta}(\Omega)$ are continuous.
Hence by complex interpolation one deduces that 
$W^{1,p}(\Omega) \subset L^s(\Omega)$, where 
$\frac{1}{s} 
= \frac{1-\theta}{r} + \frac{\theta}{q+\delta} 
< \frac{1-\theta}{r} + \frac{\theta}{q}
= \frac{1}{p}$.
Note that $s > p$ as required.
\end{proof}

Arguing as in Ziemer~\cite{Zie2} Theorem~4.4.2 one obtains a Poincar\'e inequality from the
compact inclusion $W^{1,p}(\Omega) \subset L^p(\Omega)$.

\begin{prop} \label{pbounded202}
Let $\Omega \subset \Ri^d$ be open, bounded and connected.
Let $p \in (1,\infty)$ and suppose that the inclusion $W^{1,p}(\Omega) \subset L^p(\Omega)$
is compact.
Let $\Omega_0 \subset \Omega$ be measurable and suppose that the Lebesgue measure 
$|\Omega_0| > 0$.
Then there exists a $c > 0$ such that 
\[
\|u\|_p \leq c \, \|\nabla u\|_{p}
\]
for all $u \in W^{1,p}(\Omega)$ with $\int_{\Omega_0} u = 0$.
\end{prop}
\begin{proof}
Suppose not.
Then for all $n \in \Ni$ there exists a $u_n \in W^{1,p}(\Omega)$ such that 
$\|u_n\|_p > n \, \|\nabla u_n\|_p$ and $\int_{\Omega_0} u_n = 0$.
Without loss of generality $\|u_n\|_p = 1$ for all $n \in \Ni$.
Then $\|\nabla u_n\|_p \leq \frac{1}{n}$.
So the sequence $(u_n)_{n \in \Ni}$ is bounded in $W^{1,p}(\Omega)$.
Passing to a subsequence if necessary there exists a $u \in W^{1,p}(\Omega)$
such that $\lim u_n = u$ weakly in $W^{1,p}(\Omega)$.
Then $\lim u_n = u$ strongly in $L^p(\Omega)$ and $\int_{\Omega_0} u = 0$.
Moreover $\|u\|_p = 1$ and $u \neq 0$.
Next $\|\nabla u\|_p \leq \liminf_{n \to \infty} \|\nabla u_n\|_p = 0$.
Since $\Omega$ is connected it follows that $u$ is constant by \cite{Zie2} Corollary~2.1.9.
Because $\int_{\Omega_0} u = 0$ and $|\Omega_0| > 0$ one deduces that $u = 0$.
This is a contradiction.
\end{proof}

If $\Omega \subset \Ri^d$ is a bounded open set
and $p \in (1,\infty)$, then we define 
\[
W^{1,p}_\perp(\Omega)
= \Bigl\{ u \in W^{1,p}(\Omega) \colon \int_\Omega u = 0 \Bigr\} 
 .  \]
It follows from Proposition~\ref{pbounded202} that 
$W^{1,p}_\perp(\Omega)$ equipped with the norm
$u \mapsto \|\nabla u\|_p$ is a Banach space.

\begin{cor} \label{cbounded203}
Let $\Omega \subset \Ri^d$ be open, bounded and connected.
Let $p \in (1,\infty)$ and suppose that the inclusion $W^{1,p}(\Omega) \subset L^p(\Omega)$
is compact.
Define $\altnorm{{}\cdot{}} \colon W^{1,p}(\Omega) \to [0,\infty)$ by 
$\altnorm{u} = \|\nabla u\|_p + \big| \int_\Omega u \big|$.
Then one has the following.
\begin{tabel} 
\item \label{cbounded203-1}
The function $\altnorm{{}\cdot{}}$ is a norm on $W^{1,p}(\Omega)$ which is equivalent
to $\|\cdot\|_{W^{1,p}(\Omega)}$.
\item \label{cbounded203-2}
  The map
  \begin{equation*}
    P \colon u \mapsto  u - \tfrac{1}{|\Omega|} \int_\Omega u 
  \end{equation*}
  is a projection from $W^{1,p}(\Omega)$ onto $W^{1,p}_\perp(\Omega)$. In particular,
\[
u \mapsto \Big(\tfrac{1}{|\Omega|} \int_\Omega u, u - \tfrac{1}{|\Omega|} \int_\Omega u
\Big)
\]
is a topological isomorphism from $W^{1,p}(\Omega)$ onto 
$\Ci \oplus W^{1,p}_\perp(\Omega)$.
\end{tabel}
\end{cor}
\begin{proof}
By Proposition~\ref{pbounded202}
there exists a $c > 0$ such that 
$\|u\|_p \leq c \, \|\nabla u\|_p$ for all $u \in W^{1,p}_\perp(\Omega)$.
If $u \in W^{1,p}(\Omega)$, then 
\begin{eqnarray*}
\|u\|_p
& \leq & \bigl\|u - \tfrac{1}{|\Omega|} \int_\Omega u\bigr\|_p
   + \bigl\|\tfrac{1}{|\Omega|} \int_\Omega u\bigr\|_p  \\
& \leq & c \, \bigl\|\nabla \Big( u - \tfrac{1}{|\Omega|} \int_\Omega u \Big) \bigr\|_p
   + |\Omega|^{-1+\frac{1}{p}} \Big| \int_\Omega u \Big|  \\
& = & c \, \|\nabla u\|_p + |\Omega|^{-1+\frac{1}{p}} \Big| \int_\Omega u \Big| 
\end{eqnarray*}
and the lemma follows easily.
\end{proof}

\begin{prop} \label{pbounded204}
Let $\Omega \subset \Ri^d$ be open, bounded and connected.
Let $p \in (1,\infty)$ and suppose that the inclusion $W^{1,p}(\Omega) \subset L^p(\Omega)$
is compact.
Then for all $T \in (W^{1,p}(\Omega))^*$ there exist $\kappa \in \Ci$ and 
$f_1,\ldots,f_d \in L^q(\Omega)$ such that 
\[
\langle T,u \rangle_{(W^{1,p}(\Omega))^* \times W^{1,p}(\Omega)}
= \kappa \int_\Omega \overline u + \sum_{j=1}^d \int_\Omega f_j \, \overline{\partial_j u}
\]
for all $u \in W^{1,p}(\Omega)$, where $q$ is the dual exponent of~$p$.
\end{prop}
\begin{proof}
Using Corollary~\ref{cbounded203}\ref{cbounded203-2}
it suffices to show that for all $S \in (W^{1,p}_\perp(\Omega))^*$
there exist $f_1,\ldots,f_d \in L^q(\Omega)$ such that 
\[
\langle S,u \rangle_{(W^{1,p}_\perp(\Omega))^* \times W^{1,p}_\perp(\Omega)}
= \sum_{j=1}^d \int_\Omega f_j \, \overline{\partial_j u}
\]
for all $u \in W^{1,p}_\perp(\Omega)$, where $u \mapsto \|\nabla u\|_p$ is 
the norm on $W^{1,p}_\perp(\Omega)$.
Consider the subspace $M = \{ \nabla u \colon u \in W^{1,p}_\perp(\Omega) \} $
in $L^p(\Omega)^d$.
Define $F \colon M \to \Ci$
by $F(\nabla u) = S u$.
Then $F$ is well-defined and continuous.
Therefore by Hahn--Banach there exists an extension 
$\widetilde F \in (L^p(\Omega)^d)^*$ of $F$.
The rest of the proof is straight forward.
\end{proof}

\section{Proof of Theorem~\ref{tbounded101}} \label{boundedS3}

In this section we prove Theorem~\ref{tbounded101}.
Let $\Omega \subset \Ri^d$ be a bounded connected open set.
Let $\mu \colon \Omega \to \Ri^{d \times d}$ be a bounded measurable 
function.
We suppose that $\mu$ is {\bf elliptic}, that is there exists a $\nu > 0$ such that 
\[
\RRe \sum_{k,\ell=1}^d \mu(x) \, \xi_k \, \overline{\xi_\ell} 
\geq \nu \, |\xi|^2
\]
for all $\xi \in \Ci^d$ and almost all $x \in \Omega$. 
Let $r > 2$ and suppose that $W^{1,2}(\Omega) \subset L^r(\Omega)$.

Define $\ca \colon W^{1,2}(\Omega) \to (W^{1,2}(\Omega))^*$ by 
\[
\ca(u,v) 
= \int_\Omega \mu \nabla u \cdot \overline{\nabla v}
 .  \]
Recall that 
$W^{1,p}_\perp(\Omega) = \bigl\{ u \in W^{1,p}(\Omega) \colon \int_\Omega u = 0 \bigr\} $
for all $p \in (1,\infty)$.
If $q \in (1,\infty)$ then we define 
\[
W^{-1,q}_\emptyset(\Omega) 
= \bigl(W^{1,p}(\Omega)\bigr)^*
,  \]
where $p$ is the dual exponent of~$q$.
Moreover, we define 
\[
W^{-1,q}_\perp(\Omega) = \bigl\{ T \in W^{-1,q}_\emptyset(\Omega) \colon T(\one) = 0 \bigr\} 
 .  \]
Clearly $\ca u \in W^{-1,2}_\perp(\Omega)$ for all $u \in W^{1,2}(\Omega)$
and $\ker \ca = \Ci \one$ since $\Omega$ is connected.
Define $\ca_\perp \colon W^{1,2}_\perp(\Omega) \to W^{-1,2}_\perp(\Omega)$ by 
$\ca_\perp u = \ca u$.
Then $\ca_\perp$ is injective. 
We next show that it is also surjective and
$W^{-1,2}_\perp(\Omega) = (W^{1,2}_\perp(\Omega))^*$, up to isomorphy.

\begin{prop} \label{pbounded205}
The map $\ca_\perp$ is a topological isomorphism.
\end{prop}
\begin{proof}
Define the form $\gotb \colon W^{1,2}_\perp(\Omega) \times W^{1,2}_\perp(\Omega) \to \Ci$
by 
\[
\gotb(u,v) = \int_\Omega \mu \nabla u \cdot \overline{\nabla v}
 .  \]
Then $\gotb$ is a continuous coercive sesquilinear form
by Lemma~\ref{lbounded201} and Proposition~\ref{pbounded202}.
Let $\cb \colon W^{1,2}_\perp(\Omega) \to (W^{1,2}_\perp(\Omega))^*$
be such that 
$\gotb(u,v) = \langle \cb u,v \rangle_{(W^{1,2}_\perp(\Omega))^* \times W^{1,2}_\perp(\Omega)}$
for all $u,v \in W^{1,2}_\perp(\Omega)$.
Then $\cb$ is surjective by the Lax--Milgram theorem.
Let $T \in W^{-1,2}_\perp(\Omega)$.
Then $T \in W^{-1,2}_\emptyset(\Omega) = (W^{1,2}(\Omega))^*$. 
Let $\widetilde T = T|_{W^{1,2}_\perp(\Omega)}$.
Then $\widetilde T \in (W^{1,2}_\perp(\Omega))^*$.
Hence there is a $u \in W^{1,2}_\perp(\Omega)$ such that $\cb u = \widetilde T$.
If $v \in W^{1,2}_\perp(\Omega)$, then
\[
\langle \ca u,v \rangle_{W^{-1,2}_\emptyset(\Omega) \times W^{1,2}(\Omega)}
= \gotb(u,v)
= \langle \cb u,v \rangle_{(W^{1,2}_\perp(\Omega))^* \times W^{1,2}_\perp(\Omega)}
= \widetilde T(v)
= T(v)
 .  \]
Since 
$\langle \ca u,\one \rangle_{W^{-1,2}_\emptyset(\Omega) \times W^{1,2}(\Omega)}
= 0 = T(\one)$ it follows by linearity and Corollary~{\ref{cbounded203}\ref{cbounded203-2}} that $\ca_\perp u = \ca u = T$.
\end{proof}

As a main tool for the proof of Theorem~\ref{tbounded101} we need 
truncations of Sobolev functions, which we consider next.

For all $u \in W^{1,2}(\Omega,\Ri)$ and $k \in [0,\infty)$ define 
$\zeta_{u,k} = (\sgn u) \, (|u| - k)^+$.
If no confusion is possible then we write $\zeta_k = \zeta_{u,k}$.
Moreover, define $A_k = \{ x \in \Omega \colon |u(x)| > k \} = [|u| > k]$.

\begin{lemma} \label{lbounded206}
Let $u \in W^{1,2}(\Omega,\Ri)$.
Then one has the following.
\begin{tabel}
\item \label{lbounded206-1}
$\zeta_k \in W^{1,2}(\Omega)$ for all $k \in [0,\infty)$.
\item \label{lbounded206-2}
$\one_{A_k} \, D_j u = \one_{A_k} \, D_j \zeta_k$ for all 
$j \in \{ 1,\ldots,d \} $ and $k \in [0,\infty)$.
\item \label{lbounded206-3}
The map $k \mapsto \zeta_k$ is continuous from $[0,\infty)$ into $W^{1,2}(\Omega)$.
\item \label{lbounded206-4}
If $k \in [0,\infty)$, then the map
$v \mapsto \zeta_{v,k}$ is continuous from 
$W^{1,2}(\Omega,\Ri)$ into $W^{1,2}(\Omega)$.
\end{tabel}
\end{lemma}
\begin{proof}
`\ref{lbounded206-1}' and `\ref{lbounded206-2}'.
Note that $\zeta_k = (u^+ - k)^+ - (u^- - k)^+$.
Then the statements follow from~\cite{GT} Lemma~7.6.

`\ref{lbounded206-3}'.
This follows from the Lebesgue dominated convergence theorem.

`\ref{lbounded206-4}'.
This follows from~\ref{lbounded206-1} and~\cite{MaM} Theorem~1.
\end{proof}

A key estimate for the proof of Theorem~\ref{tbounded101} is the next lemma.

\begin{lemma} \label{lbounded207}
Let $u \in W^{1,2}_\perp(\Omega,\Ri)$.
Then there exists a $\gamma \geq 0$ such that 
$\|\zeta_k\|_2 \leq \gamma \, \|\nabla \zeta_k\|_2$
for all $k \in [0,\infty)$.
\end{lemma}
\begin{proof}
We split the proof into two cases depending whether $u$ is bounded or not.

{\bf Case 1.~}~Suppose $u$ is unbounded.
If $k \in [0,\infty)$ and $\|\nabla \zeta_k\|_2= 0$, then 
$\zeta_k$ is constant and consequently $u$ is bounded, which is a 
contradiction.
Hence $\|\nabla \zeta_k\|_2\neq 0$ for all $k \in [0,\infty)$.
Since both $k \mapsto \|\zeta_k\|_2$ and 
$k \mapsto \|\nabla \zeta_k\|_2$ are continuous on
$[0,\infty)$ by Lemma~\ref{lbounded206}\ref{lbounded206-3}, 
it suffices to show that 
\begin{equation}
\limsup_{k \to \infty} 
   \frac{ \|\zeta_k\|_2 }
        { \|\nabla \zeta_k\|_2 }
     \leq 1 .
\label{elbounded207;1}
\end{equation}
Suppose that (\ref{elbounded207;1}) is false.
Then there exists a sequence $(k_n)_{n \in \Ni}$ in $\Ri$ such that 
$k_n \geq n$ for all $n \in \Ni$ and 
$\|\zeta_{k_n}\|_2 > \|\nabla \zeta_{k_n}\|_2$ for all $n \in \Ni$.
Define $v_n = \|\zeta_{k_n}\|_2^{-1} \, \zeta_{k_n}$
for all $n \in \Ni$.
Then $v_n \in W^{1,2}(\Omega)$, $\|v_n\|_2 = 1$
and $\|\nabla v_n\|_2 \leq 1$ for all $n \in \Ni$.
So the sequence $(v_n)_{n \in \Ni}$ is bounded in $W^{1,2}(\Omega)$.
Passing to a subsequence if necessary we may assume that there is a 
$v \in W^{1,2}(\Omega)$ such that $\lim v_n = v$ weakly in $W^{1,2}(\Omega)$.
Then $\lim v_n = v$ in $L^2(\Omega)$.
So $\|v\|_2 = 1$ and in particular $v \neq 0$.
But $v(x) = \lim_{n \to \infty} v_n(x) = 0$ for almost every $x \in \Omega$.
This is a contradiction.

{\bf Case 2.~}~Suppose $u$ is bounded.
Without loss of generality we may assume that $u \neq 0$.
Let $k \in [0,\|u\|_\infty)$ and suppose that 
$\|\nabla \zeta_k\|_2 = 0$.
Then $\zeta_k$ is constant, say~$\delta$.
If $\delta = 0$, then $|u| \leq k$ a.e.,
which is not possible since $k < \|u\|_\infty$.
Suppose $\delta > 0$.
Note that $\zeta_k(x) \leq 0 < \delta$ for all $x \in \Omega$ with $u(x) \leq k$.
So $u(x) = k + \delta$ for all $x \in \Omega$.
But then $\int_\Omega u \neq 0$.
Similarly $\delta < 0$ gives a contradiction.
Hence $\|\nabla \zeta_k\|_2 \neq 0$ for 
all $k \in [0,\|u\|_\infty)$.

Arguing as in Case 1 and using Lemma~\ref{lbounded206}\ref{lbounded206-3}
it follows that for all $k_1 \in (0,\|u\|_\infty)$
there exists a $c_1 > 0$ such that 
$\|\zeta_k\|_2 \leq c_1 \, \|\nabla \zeta_k\|_2$
for all $k \in [0,k_1]$.

Finally we show that there exist $k_0 \in (0,\|u\|_\infty)$
and $c_0 > 0$ such that 
$\|\zeta_k\|_2 \leq c_0 \, \|\nabla \zeta_k\|_2$
for all $k \in (k_0,\infty)$.
If $|u| = \|u\|_\infty$ a.e., then 
$| [u = \|u\|_\infty ]|
   = \frac{1}{2} \, |\Omega| > 0$,
where we use that $\int_\Omega u = 0$.
Then $w = \one_{[u = \|u\|_\infty]}\,u = u \vee 0 \in W^{1,2}(\Omega)$.
Using~\cite{GT} Lemma~7.7 we deduce that $\nabla w = 0$ a.e.\ and
this implies that $|[u = \|u\|_\infty]| \in \{ 0,|\Omega| \} $, 
which is a contradiction.
Hence there is a $k_0 \in (0,\|u\|_\infty)$ such that 
$|[ |u| \leq k_0 ]| > 0$.
Write $\Omega_0 = [ |u| \leq k_0 ]$.
By Lemma~\ref{lbounded201} and 
Proposition~\ref{pbounded202} there exists a $c_0 > 0$ such that 
$\|v\|_2 \leq c_0 \, \|\nabla v\|_2$
for all $v \in W^{1,2}(\Omega)$ with $\int_{\Omega_0} v = 0$.
If $k \in (k_0,\infty)$, then $\zeta_k(x) = 0$ for all 
$x \in \Omega_0$, so $\int_{\Omega_0} \zeta_k = 0$.
Hence $\|\zeta_k\|_2 \leq c_0 \, \|\nabla \zeta_k\|_2$.
\end{proof}

For all $u \in W^{1,2}_\perp(\Omega,\Ri)$
define $\gamma_u \in [0,\infty)$ to be the minimum of all 
$\gamma \geq 0$ such that 
$\|\zeta_k\|_2 \leq \gamma \, \|\nabla \zeta_k\|_2$
for all $k \in [0,\infty)$. 
Recall that $r > 2$ is such that $W^{1,2}(\Omega) \subset L^r(\Omega)$.

\begin{prop} \label{pbounded301}
Let $u \in W^{1,2}_\perp(\Omega,\Ri)$ and 
$q > d$ with $\frac{1}{2}-\frac{1}{q} > \frac{1}{r}$.
Further let $f_1,\ldots,f_d \in L^q(\Omega)$
and suppose that 
$\langle \ca u,v\rangle_{W^{-1,2}_\emptyset(\Omega) \times W^{1,2}(\Omega)} 
= \sum_{j=1}^d (f_j, \partial_j v)_2$
for all $v \in W^{1,2}(\Omega)$.
Then $u \in L^\infty(\Omega)$.
Moreover
\[
\|u\|_\infty
\leq 2^{(\frac{1}{2}-\frac{1}{q})/\delta} \, \frac{E}{\nu} \, \sqrt{\bigl(1 +
  \gamma_u^2\bigr)|\Omega|^{\delta}} \, \left(\sum_{j=1}^d \bigl\|f_j\bigr\|_q^2\right)^{\frac12}
,  \]
where
$\delta = \frac{1}{2} - \frac{1}{q} - \frac{1}{r} > 0$
and $\nu$ is the ellipticity constant of $\mu$.
Finally, $E > 0$ is such that 
$\|v\|_r \leq E \, \|v\|_{W^{1,2}(\Omega)}$ for all 
$v \in W^{1,2}(\Omega)$.
\end{prop}
\begin{proof}
For all $k \in [0,\infty)$ define 
$\zeta_k = (\sgn u) \, (|u| - k)^+ \in W^{1,2}(\Omega)$ 
and $A_k = [|u| > k] $ as before.
Let $k \in [0,\infty)$.
Then 
\begin{eqnarray*}
\nu \, \bigl\|\nabla \zeta_k\bigr\|_2^2
& \leq & \int_\Omega \mu \nabla \zeta_k \cdot \nabla \zeta_k  = \int_\Omega \mu \nabla u \cdot \nabla \zeta_k  = \sum_{j=1}^d \int_{A_k} f_j \, \partial_j \zeta_k  \\
& \leq & \Bigl( \sum_{j=1}^d \int_{A_k} |f_j|^2 \Bigr)^{1/2} 
          \bigl\|\nabla \zeta_k\bigr\|_2 \\
& \leq & \frac{\nu}{2} \, \bigl\|\nabla \zeta_k\bigr\|_2^2
   + \frac{1}{2\nu} \, \sum_{j=1}^d \int_{A_k} |f_j|^2
 .
\end{eqnarray*}
Hence 
\[
\bigl\|\nabla \zeta_k\bigr\|_2^2
\leq \frac{1}{\nu^2} \sum_{j=1}^d \int_{A_k} |f_j|^2
\leq \frac{|A_k|^{1 - \frac{2}{q}}}{\nu^2} \, \sum_{j=1}^d \bigl\|f_j\bigr\|_q^2
 .  \]
By assumption $W^{1,2}(\Omega) \subset L^r(\Omega)$.
Then 
\begin{eqnarray*}
\left( \int_{A_k} \bigl( |u| - k \bigr)^{r} \right)^{\frac{2}{r}}
& = & \bigl\|\zeta_k\bigr\|_{L^{r}(\Omega)}^2  \\
& \leq & E^2 \, \bigl\|\zeta_k\bigr\|_{W^{1,2}(\Omega)}^2
= E^2 \, \Bigl(\bigl\|\zeta_k\bigr\|_2^2 + \bigl\|\nabla \zeta_k\bigr\|_2^2 \Bigr)  \\
& \leq & E^2 \, \bigl(1 + \gamma_u^2\bigr) \, 
   \frac{|A_k|^{1 - \frac{2}{q}}}{\nu^2} \, \sum_{j=1}^d \bigl\|f_j\bigr\|_q^2
 .
\end{eqnarray*}
Next let $h,k \in [0,\infty)$ with $h > k$.
Then $A_h \subset A_k$ and 
\[
(h-k)^2 \, |A_h|^{\frac{2}{r}}
\leq \left( \int_{A_h} \bigl| |u| - k \bigr|^{r} \right)^{\frac{2}{r}}  
\leq \left( \int_{A_k} \bigl| |u| - k \bigr|^{r} \right)^{\frac{2}{r}}  
\leq E^2 \, \bigl(1 + \gamma_u^2\bigr) \, 
   \frac{|A_k|^{1 - \frac{2}{q}}}{\nu^2} \, \sum_{j=1}^d \bigl\|f_j\bigr\|_{q}^2
 .  \]
Equivalently
\[
|A_h|
\leq \frac{1}{(h-k)^r} \, \Big( \frac{E}{\nu} \Big)^{r} 
     \bigl(1 + \gamma_u^2\bigr)^{\frac{r}{2}} \, 
     \biggl( \sum_{j=1}^d \bigl\|f_j\bigr\|_{q}^2 \biggr)^{\frac{r}{2}} \, 
     |A_k|^{(1-\frac{2}{q})\frac{r}{2}}
.  \]
Due to $(1-\frac{2}{q})\frac{r}{2} = (\frac{1}{2} - \frac{1}{q}) r > 1$ by assumption, 
it now follows from the Stampacchia lemma (\cite{KinS} Chapter~II, Appendix~B, Lemma~2.1) that 
$u \in L^\infty(\Omega)$ and 
\[
\|u\|_\infty
\leq 2^{(\frac{1}{2}-\frac{1}{q})/\delta} \, \frac{E}{\nu} \, 
     \sqrt{\bigl(1 +
  \gamma_u^2\bigr)|\Omega|^{\delta}} \, 
         \left(\sum_{j=1}^d \bigl\|f_j\bigr\|_q^2\right)^{\frac12}
 .  \]
This completes the proof of the proposition.
\end{proof}

\begin{proof}[{\bf Proof of Theorem~\ref{tbounded101}.}]
Let $u \in W^{1,2}(\Omega)$ be
such that $\ca u \in (W^{1,p}(\Omega))^*$, where~$p$ is the dual 
exponent of $q$.
By Lemma~\ref{lbounded201} the inclusion $W^{1,p}(\Omega) \subset L^p(\Omega)$ is compact.
Hence by Proposition~\ref{pbounded204} there exist $\kappa \in \Ci$ and 
$f_1,\ldots,f_d \in L^q(\Omega)$ such that 
\[
\langle \ca u,v \rangle_{(W^{1,p}(\Omega))^* \times W^{1,p}(\Omega)}
= \kappa \int_\Omega \overline v + \sum_{j=1}^d \int_\Omega f_j \, \overline{\partial_j v}
\]
for all $v \in W^{1,p}(\Omega)$.
Choosing $v = \one$ one deduces that $\kappa = 0$ and $\ca u \in W^{-1,2}_\perp(\Omega)$.
Without loss of generality we may assume that $u \in W^{1,2}_\perp(\Omega)$.
Moreover, we may also assume that $u$ is real valued.
Now apply Proposition~\ref{pbounded301} to obtain $u \in L^\infty(\Omega)$.

If we start with $T \in W^{-1,q}_\perp(\Omega)$, then there exists a unique
$u \in W^{1,2}_\perp(\Omega)$ such that $\ca u = T$ by Proposition~\ref{pbounded205}.
Then $\ca u \in W^{-1,q}_\perp(\Omega) \subset (W^{1,p}(\Omega))^*$, so
$u \in L^\infty(\Omega)$ by the above.

For the estimate it suffices to show that the map $T \mapsto u$ has closed graph
in the space $W^{-1,q}_\perp(\Omega) \times L^\infty(\Omega)$.
Let $T,T_1,T_2,\ldots \in W^{-1,q}_\perp(\Omega)$ and $u \in L^\infty(\Omega)$.
Suppose that $\lim T_n = T$ in $W^{-1,q}_\perp(\Omega)$
and $\lim (\ca_\perp)^{-1} T_n = u$ in $L^\infty(\Omega)$.
Then $\lim T_n = T$ in $W^{-1,2}_\perp(\Omega)$, 
so $\lim (\ca_\perp)^{-1} T_n = (\ca_\perp)^{-1} T$ in $W^{1,2}_\perp(\Omega)$
and hence also in $L^2(\Omega)$.
But $\lim (\ca_\perp)^{-1} T_n = u$ in $L^\infty(\Omega)$ and 
therefore also in $L^2(\Omega)$.
Consequently $(\ca_\perp)^{-1} T = u$ as required.
\end{proof}

\section{Interpolation and maximal Sobolev regularity}\label{boundedS4}

In this section, we use the structure of $W^{1,p}_\perp(\Omega)$ as a complemented subspace
of $W^{1,p}(\Omega)$ to establish interpolation results.
Optimal Sobolev
regularity for the pure Neumann operator $\ca_\perp$ for $p$ close to $2$ also follows.
This is
particularly interesting for space dimension $d=2$.
The first step is to show that
$W^{1,p}_\perp(\Omega)$ and $W^{-1,p}_\perp(\Omega)$ form an interpolation scale with
respect to $p$.
\begin{prop}
  \label{pbounded401}
  Let $\Omega \subset \R^d$ be open and bounded.
Let $p_0,p_1 \in (1,\infty)$,
  $\theta \in (0,1)$ and set $\frac1p = \frac{1-\theta}{p_0} + \frac\theta{p_1}$.
Then
  \begin{equation*}
    \bigl[W^{1,p_0}_\perp(\Omega),W^{1,p_1}_\perp(\Omega)\bigr]_\theta =
    \bigl(W^{1,p_0}_\perp(\Omega),W^{1,p_1}_\perp(\Omega)\bigr)_{\theta,p} = W^{1,p}_\perp(\Omega)
  \end{equation*}
  and
  \begin{equation*}
        \bigl[W^{-1,p_0}_\perp(\Omega),W^{-1,p_1}_\perp(\Omega)\bigr]_\theta =
    \bigl(W^{-1,p_0}_\perp(\Omega),W^{-1,p_1}_\perp(\Omega)\bigr)_{\theta,p} = W^{-1,p}_\perp(\Omega).
  \end{equation*}
\end{prop}

\begin{proof}
It follows from~\eqref{ebounded1} that 
\[
\bigl[W^{1,p_0}(\Omega),W^{1,p_1}(\Omega)\bigr]_\theta = W^{1,p}(\Omega)
 .  \]
Arguing as in~\eqref{ebounded1}, but using the reiteration theorem 
for real interpolation~\cite{BL}, Theorem~3.5.3 one deduces similarly
\begin{equation*}
  \bigl(W^{1,p_0}(\Omega),W^{1,p_1}(\Omega)\bigr)_{\theta,p} = W^{1,p}(\Omega).
\end{equation*}
Note that for all $r \in (1,\infty)$ the projection $P$ in
Corollary~\ref{cbounded203}\ref{cbounded203-2} 
maps $W^{1,r}(\Omega)$ onto $W^{1,r}_\perp(\Omega)$, so
$W^{1,r}_\perp(\Omega)$ is a complemented subspace of $W^{1,r}(\Omega)$.
We further observe that 
$W^{1,p_i}_\perp(\Omega) = W^{1,p_i}(\Omega) \cap W^{1,\min(p_0,p_1)}_\perp(\Omega)$ for
$i=1,2$.
Thus, interpolation theory for complemented subspaces (\cite{Tri}
Theorem~1.17.1.1) shows that
\begin{equation*}
 \bigl[W^{1,p_0}_\perp(\Omega),W^{1,p_1}_\perp(\Omega)\bigr]_\theta =
  \bigl(W^{1,p_0}_\perp(\Omega),W^{1,p_1}_\perp(\Omega)\bigr)_{\theta,p} =
  W^{1,p}(\Omega) \cap W^{1,\min(p_0,p_1)}_\perp(\Omega) = W^{1,p}_\perp(\Omega).
\end{equation*}
Concerning the dual spaces, it is easy to see that for all $q \in (1,\infty)$ the 
operator $T \mapsto T - \frac{1}{|\Omega|} \, \langle T, \one \rangle \one$
is a projection from $W^{-1,q}(\Omega)$ onto $W^{-1,q}_\perp(\Omega)$.
Hence the assertion follows with the same argument and the
duality properties of the real and complex 
interpolation functors, see~\cite{Tri} Subsections~1.11.2 and~1.11.3.
\end{proof}

The first result
derived from Proposition~\ref{pbounded401} together with Theorem~\ref{tbounded101}
is the following mapping property for $\ca_\perp^{-1}$ on the 
$W^{-1,p}_\perp(\Omega)$ spaces for all
$p > 2$.
Note that we do not require that $p > d$.

\begin{cor} \label{cbounded402}
Let $\Omega \subset \Ri^d$ be a bounded connected open set.
Let $r \in (2,\infty)$ and
suppose that $W^{1,2}(\Omega) \subset L^r(\Omega)$.
Let further
$q \in (d,\infty)$ and suppose that $\frac{1}{2}-\frac{1}{q} > \frac{1}{r}$.
Let $p \in (2,q)$.
Let $\mu \colon \Omega \to \Ri^{d \times d}$ be a bounded measurable 
elliptic function and let $\ca \colon W^{1,2}(\Omega) \to (W^{1,2}(\Omega))^*$
be the associated operator.
Then $\ca_\perp^{-1}$ maps $W^{-1,p}_\perp(\Omega)$ into $L^s(\Omega)$,
where $\frac{1}{s} = \frac{1-\theta}{r}$ and 
$\theta \in (0,1)$ is such that 
$\frac{1}{p} = \frac{1-\theta}{2} + \frac{\theta}{q}$.
\end{cor}

\begin{proof}
The operator $\ca_\perp^{-1}$ maps $W^{-1,2}_\perp(\Omega)$ continuously into 
$W^{1,2}_\perp(\Omega) \subset L^r(\Omega)$ by Proposition~\ref{pbounded205}.
Moreover, $\ca_\perp^{-1}$ maps $W^{-1,q}_\perp(\Omega)$ continuously into 
$L^\infty(\Omega)$ by Theorem~\ref{tbounded101}.
Now use complex interpolation and Proposition~\ref{pbounded401}.
\end{proof}

Due to Proposition~\ref{pbounded401} and the work from the previous sections, a
maximal Sobolev regularity result for $p$ close to $2$ follows by an application of the
{\v{S}}ne{\u{\i}}berg stability theorem. 

\begin{thm} \label{tbounded402}
Let $\Omega \subset \Ri^d$ be a bounded connected open set.
Let $r \in (2,\infty)$ and
suppose that $W^{1,2}(\Omega) \subset L^r(\Omega)$.
Let $\mu \colon \Omega \to \Ri^{d \times d}$ be a bounded measurable 
elliptic function and let $\ca \colon W^{1,2}(\Omega) \to (W^{1,2}(\Omega))^*$
be the associated operator.
Then there exists a
$\delta > 0$ such that $\ca_\perp$ is a topological isomorphism between
$W^{1,p}_\perp(\Omega)$ and $W^{-1,p}_\perp(\Omega)$ for all $p \in (2-\delta,2+\delta)$.
\end{thm}

\begin{proof}
Under the assumptions, $\ca_\perp$ is a topological isomorphism between
$W^{1,2}_\perp(\Omega)$ and $W^{-1,2}_\perp(\Omega)$ by
Proposition~\ref{pbounded205}.
Proposition~\ref{pbounded401} shows that these spaces are
simultaneous interpolation spaces in the $W^{1,p}_\perp(\Omega)$ and
$W^{-1,p}_\perp(\Omega)$ scale.
The {\v{S}}ne{\u{\i}}berg stability
theorem~\cite{Sneiberg} implies that there is a $\delta > 0$ such that $\ca_\perp$
remains an isomorphism between $W^{1,p}_\perp(\Omega)$ and $W^{-1,p}_\perp(\Omega)$ for 
all $p \in (2-\delta,2+\delta)$.
\end{proof}

There exist quantitative results on the size of $\delta$ derived
from the {\v{S}}ne{\u{\i}}berg result in
Theorem~\ref{tbounded402}.
 We refer to~\cite{ABES}, Appendix~A.
The most crucial information is that one can choose $\delta$ to
depend only on the ellipticity constant and the upper bound
$\|\mu\|_\infty$ of the coefficient function $\mu$
of~$\ca$.
Moreover, for all $p \in (2-\delta,2+\delta)$, the
operator norm
$\|\ca_\perp^{-1}\|_{W^{-1,p}_\perp(\Omega)\to W^{1,p}_\perp(\Omega)}$ 
can be estimated by a multiple of
$\|\ca_\perp^{-1}\|_{W^{-1,2}_\perp(\Omega)\to W^{1,2}_\perp(\Omega)}$.
By Lax-Milgram, the latter can be estimated by $1/\nu$, where $\nu$ is
the ellipticity constant of~$\mu$.

Theorem~\ref{tbounded402} yields further corollaries
for $d=2$.

\begin{cor} \label{cor:cbounded403}
Adopt the notation and assumptions of Theorem~\ref{tbounded402}.
Let $d=2$.
Let $q \in (2,2+\delta)$ and suppose that $\frac12 - \frac1q > \frac1r$.
Then $W^{1,s}(\Omega) \subset L^\infty(\Omega)$ for all $s \geq q$.
\end{cor}
\begin{proof}
It follows from Theorem~\ref{tbounded101} that 
$\ca_\perp^{-1} W^{-1,q}_\perp(\Omega) \subset L^\infty(\Omega)$.
But $\ca_\perp^{-1} W^{-1,q}_\perp(\Omega) = W^{1,q}_\perp(\Omega)$
by Theorem~\ref{tbounded402}.
Since $W^{1,q}(\Omega) = W^{1,q}_\perp(\Omega) + \Ci \, \one$
the corollary follows.
\end{proof}

The parameter $\delta$ in the previous corollary depends on the
coefficient function $\mu$
via the {\v{S}}ne{\u{\i}}berg theorem. 
If $\Omega$ is smooth enough so that 
the full Sobolev embedding for $W^{1,2}(\Omega)$ is available, then 
no coefficient function is needed (at least in the formulation of the 
corollary).

\begin{cor} \label{cbounded405}
Let $\Omega \subset \Ri^2$ be a bounded connected open set.
Suppose that $W^{1,2}(\Omega) \subset L^r(\Omega)$ for all $r \in (2,\infty)$.
Then $W^{1,s}(\Omega) \subset L^\infty(\Omega)$ for all $s \in (2,\infty)$.
\end{cor}
\begin{proof}
Choose $\mu = I$.
Let $\delta > 0$ be as in Theorem~\ref{tbounded402}.
Let $s \in (2,\infty)$.
Then there exists a $q \in (2,2+\delta) \cap (2,s]$.
Now apply Corollary~\ref{cor:cbounded403}.
\end{proof}

The third corollary
concerns H\"older regularity of solutions $u$ of $\ca_\perp u = T$ with
$T \in W^{-1,q}_\perp(\Omega)$ for $q>2$ and a uniform estimate.
We do not pass through Theorem~\ref{tbounded101} for this result.
The price to pay is a Sobolev
embedding assumption for the H\"older space similar to the one in Theorem~\ref{tbounded101}.

\begin{cor} \label{cbounded406}
Let $\Omega \subset \Ri^2$ be a bounded connected open set.
Suppose that for all $q \in (2,\infty)$ there exists an $\alpha \in (0,1)$
such that $W^{1,q}(\Omega) \subset C^\alpha(\overline\Omega)$.
Let $\mu \colon \Omega \to \Ri^{d \times d}$ be a bounded measurable 
elliptic function and let $\ca \colon W^{1,2}(\Omega) \to (W^{1,2}(\Omega))^*$
be the associated operator.
Then one has the following.
\begin{tabel}
\item \label{cbounded406-1}
For all $q \in (2,\infty)$ there exists an $\alpha \in (0,1)$
such that $\ca_\perp^{-1} W^{-1,q}_\perp(\Omega) \subset C^\alpha(\overline\Omega)$.
\item \label{cbounded406-2}
For all $q \in (2,\infty)$ and $R > 0$ the set 
\[ 
\bigl\{ \ca_\perp^{-1}(T) : T \in W^{-1,q}_\perp(\Omega) \mbox{ and }
      \|T\|_{W^{-1,q}_{\emptyset}(\Omega)} \leq R \bigr\}
\]
is compact in $C(\overline \Omega)$.
\end{tabel}
\end{cor}
\begin{proof} 
`\ref{cbounded406-1}'.
Since $C^\alpha(\overline\Omega) \subset L^\infty(\Omega) \subset L^r(\Omega)$
for all $\alpha \in (0,1)$ and $r \in (1,\infty)$, it follows from 
Lemma~\ref{lbounded201} that there exists an $r \in (2,\infty)$ such that 
$W^{1,2}(\Omega) \subset L^r(\Omega)$.
Let $\delta > 0$ be as in Theorem~\ref{tbounded402}.
Let $s \in (2,2+\delta) \cap (2,q]$.
By assumption there exists an $\alpha \in (0,1)$
such that $W^{1,s}(\Omega) \subset C^\alpha(\overline\Omega)$.
Then 
$\ca_\perp^{-1} W^{-1,q}_\perp(\Omega) 
\subset \ca_\perp^{-1} W^{-1,s}_\perp(\Omega) 
= W^{1,s}_\perp(\Omega) 
\subset C^\alpha(\overline\Omega)$.

`\ref{cbounded406-2}'.
This follows from statement~\ref{cbounded406-1} and the 
Arzel\`{a}--Ascoli theorem.
\end{proof}

The situation for the H\"older-Sobolev embedding assumption in Corollary~\ref{cbounded406}
is similar to the assumption on the Sobolev embedding in Theorem~\ref{tbounded101}.
It is satisfied for example when for all $q \in (2,\infty)$ the domain
$\Omega$ is a connected $W^{1,q}$-extension domain 
and then one can choose $\alpha = 1-2/q$, but there are also examples of
(non-extension) domains with sufficiently regular cusps where the assumption is satisfied in
the weaker form, see~\cite{AF} Theorem~4.53.
Note however that the optimal embedding for $W^{1,q}(\Omega)$
into the H\"older space of order $1-2/q$ implies the $W^{1,r}$-extension property 
for all $r > q$, see~\cite{Koskela} Theorem~A.

\subsection*{Acknowledgements}
The first and third named authors are grateful for a most stimulating 
stay at the RICAM.
Part of this work is supported by the Marsden Fund Council from Government
funding, administered by the Royal Society of New Zealand.

\small 

\noindent
{\sc A.F.M. ter Elst,
Department of Mathematics,
University of Auckland,
Private bag 92019,
Auckland 1142,
New Zealand}  \\
{\em E-mail address}\/: {\bf terelst@math.auckland.ac.nz}

\mbox{}

\noindent
{\sc H. Meinlschmidt,
Johann Radon Institute for Computational and Applied Mathematics (RICAM),
Altenberger Stra\ss e 69, 
4040 Linz, 
Austria}  \\
{\em E-mail address}\/: {\bf hannes.meinlschmidt@ricam.oeaw.ac.at}

\mbox{}

\noindent
{\sc J. Rehberg,
Weierstrass Institute for Applied Analysis and Stochastics,
Mohrenstr.~39, 
10117 Berlin, 
Germany}  \\
{\em E-mail address}\/: {\bf rehberg@wias-berlin.de}

\end{document}